\documentclass[11pt]{article}
\usepackage[margin=1.1in]{geometry}
\usepackage{enumerate}
\usepackage{amssymb}
\usepackage{amsmath}
\usepackage{mathrsfs}
\usepackage{amsthm}
\usepackage{mathtools}
\usepackage{float}
\usepackage{braket}
\usepackage{cancel}
\usepackage{eufrak}
\usepackage{algorithm} 
\usepackage{algpseudocode} 
\usepackage{array}
\usepackage{xcolor}

\usepackage{comment}
\usepackage{graphicx}
\usepackage[outercaption]{sidecap}
\graphicspath{ {./images/} }
\usepackage{sidecap}
\usepackage[backend=biber, sorting=nty, maxnames=99, minnames=99]{biblatex}
\usepackage{wrapfig}
\usepackage{tikz}
\usepackage{pgffor}

\usepackage{cleveref}
\newtheorem*{theorem*}{Theorem}
\newtheorem{theorem}{Theorem}[section]

\newtheorem{proposition}[theorem]{Proposition}
\newtheorem{corollary}[theorem]{Corollary}

\newtheorem{rem}[theorem]{Remark}
\newtheorem{claim}[theorem]{Claim}

\title{Every Poset has a Large Cut}
\author{Nati Linial\thanks{School of Computer Science and Engineering, Hebrew University, Jerusalem 91904, Israel. e-mail: nati@cs.huji.ac.il. Supported in part by ISF Grant 1150/23, "Geometric Aspects of Combinatorics".}
			\and{Ori Shoshani\thanks{School of Computer Science and Engineering, Hebrew University, Jerusalem 91904, Israel. e-mail: ori.shoshani.045@gmail.com.}}}
\date{}

\usepackage{xargs}                      
\usepackage{xcolor}
\usepackage{comment}

\usepackage[colorinlistoftodos,prependcaption,textsize=tiny]{todonotes}
\newcommandx{\Todo}[2][1=]{\todo[linecolor=blue,backgroundcolor=blue!25,bordercolor=blue,#1]{#2}}
\newcommandx{\Note}[2][1=]{\todo[linecolor=green,backgroundcolor=green!25,bordercolor=green,#1]{#2}}

\newcommandx{\revised}[2][1=]{\todo[linecolor=orange,backgroundcolor=orange!25,bordercolor=orange,#1]{#2}}

\begin{document}
\maketitle
\begin{abstract}
We prove that every finite poset has a directed cut with at least one half
of the poset's pairwise order relations. The bound is tight. Also,
the largest directed cut in a poset can be found in linear time.
\end{abstract}
\section{Statement}\label{sect:large cut}
A {\em cut} in a graph $(V,E)$ is a partition $(L,R)$ of the vertex set $V$.
Its {\em size} is the number of edges $xy\in E$ with $x\in L$
and $y\in R$. It is well known that every finite graph has a cut of 
size $\ge |E|/2$, and, moreover, that this bound is tight up to additive lower order terms.
Such large cuts can be found by using either a {\em random} partition of the vertices
or arise as a {\em locally optimal} partition. What is the situation in posets?

Let $(P,>)$ be a finite poset. 
For two subsets $X,Y$ of $P$ (not necessarily disjoint), we denote
$E(X,Y):=\{(x,y)|x\in X, y\in Y\text{~and~}x<y\}$, and $e(X,Y):=|E(X,Y)|$.
Where necessary, we specify the underlying poset and write, e.g., $e_P(X,Y)$.
A {\em cut} in $(P,>)$ is a partition $(B,U)$ of the set $P$.
Its {\em size} is defined as $e(B,U)$. Here is our main result:

\begin{theorem}\label{thm:main}
Every finite poset $(P,>)$ has a cut $(B,U)$ with
\begin{equation}\label{eq:main}
e\left(B, U\right) \ge \frac{e\left(P, P\right)}{2}.
\end{equation}
The bound is tight up to lower-order terms. Such a cut can be found in time $O(|(P,>)|)$.
\end{theorem}

Tightness is established by considering an $n$-element chain, where
$e(P,P)={n\choose 2}$, since 
\[\max e(B,U) = \lfloor\frac{n^2}{4}\rfloor=\frac{1+o_n(1)}{2}\cdot{n\choose 2},\]
where the $\max$ is over all cuts $(B,U)$ of $P$.

With an eye to the situation in graphs, random
partitions of posets yield the weaker estimate
of $1/4$. A simple local-optimality
consideration yields the better bound of 
$1/3$, but reaching the desired $1/2$
takes some work. The random construction yields $1/4$ for all digraphs,
and this is tight, as witnessed e.g.,
by random tournaments. More interestingly, as shown in \cite{alon2007maxdirectedcuts} the best
answer remains $1/4$ even for {\em acyclic} digraphs.

Also in terms of the computational complexity of this problem, posets are
better behaved than general acyclic digraphs. Indeed
as shown in \cite{lampis2011digraph}, it is NP-hard to
find a directed max-cut in acyclic directed graphs. In contrast,
the proofs below show how to find a directed max-cut of a poset $(P,>)$ 
in time linear in $|(P,>)|$.

\subsection{Preliminaries}

We start with some standard notation: An {\em ideal} in $P$
is an upward closed subset $I\subseteq P$, i.e., $x\in I$ and $y>x$ imply that $y\in I$.
Likewise, a downward closed subset $F\subseteq P$ is called a {\em filter}.
We need the following simple observation:
    
\begin{claim}\label{claim:ideal-filter}
If $(B,U)$ is a largest cut in a finite poset
$(P,>)$, then $B$ is a filter and $U$ is an ideal.
\begin{proof}
We argue by contradiction. If either $B$ is not a filter or $U$ is not an ideal,
then there are elements $x\in U, y\in B$
with $x<y$. Let us swap $x$ and $y$. Namely, let $V:=(B\setminus\{y\})\cup\{x\}$
and let $W:= (U\setminus\{x\})\cup\{y\}$. Then 
$e(V,W)=1+e(B,U)+e(B,\left\{y\right\})+e(\left\{x\right\},U)-e(B,\left\{x\right\})-e(\left\{y\right\},U)$, therefore
\[e(V,W)\ge1+e(B,U)\]
contrary to the maximality of $(B,U)$.
\end{proof}
    \end{claim}

The set $P$ naturally splits into three categories.
\begin{itemize}
    \item 
The set $\Delta=\Delta_P$ of {\em deficit} elements $v$ for which $e\left(\left\{v\right\}, P\right)>e\left(P,\left\{v\right\}\right)$.
\item 
The set $\Sigma=\Sigma_P$ of {\em surplus} elements $v$ for which $e\left(\left\{v\right\}, P\right)<e\left(P,\left\{v\right\}\right)$.
\item 
The set $\Lambda=\Lambda_P$ of {\em balanced} elements $v$ for which $e\left(\left\{v\right\}, P\right) = e\left(P,\left\{v\right\}\right)$.
\end{itemize}

Here are some simple observations concerning this partition of the elements:

\begin{proposition}\label{prop:continuity}
Let $a<b$ be two elements in $(P,>)$. 
\begin{itemize}
\item 
If $a$ is either a surplus or a balanced element, then $b$ is a surplus element.
\item 
If $b$ is either a deficit or a balanced element, then $a$ is a deficit element.
\end{itemize}
In particular, the set $\Lambda_P$ of balanced elements in $P$
forms an antichain.
\end{proposition}
\begin{proof}
This follows from the obvious relations:
\[e\left(P, \left\{b\right\}\right)>e\left(P, \left\{a\right\}\right)\ge e\left(\left\{a\right\}, P\right) > e\left(\left\{b\right\}, P\right)\]
and
\[e\left(\left\{a\right\}, P\right)> e\left(\left\{b\right\}, P\right) \ge e\left(P, \left\{b\right\}\right) > e\left(P, \left\{a\right\}\right).\]
\end{proof}

\begin{corollary}\label{cor:equiv}
If $(B,U)$ is a cut in $P$, where
$B \supseteq \Delta_P$, and $U\supseteq \Sigma_P$, then 
$e(B,U)$ does not depend on the location of the balanced elements in $P$.
\end{corollary}
\begin{proof}
By proposition \ref{prop:continuity}, for every $v\in \Lambda$
there holds $e(\Delta_P, v) = e(v, \Sigma_P)$, so that $v$'s
contribution to $e(B,U)$ is the same whether it belongs to $B$ or to $U$.
\end{proof}

\begin{proposition}\label{prop:maximality happens}
If $(B,U)$ is a largest cut in a finite poset $(P,>)$, then
\begin{enumerate}
    \item 
    $\Delta_P \subseteq B$
    \item 
    $\Sigma_P \subseteq U$
\end{enumerate}
In words, every deficit element of $P$ must be in B, and every surplus element in $U$,
while the size of the cut does not depend on where the balanced elements are.
\begin{proof}
By claim \ref{claim:ideal-filter}, $B$ is a filter and $U$ is an ideal.
Let $u\in U$ be a minimal element in $U$, and let us
move $u$ from $U$ to $B$. The size of the cut changes (additively) by 
$e\left(\left\{u\right\}, U\right)-e\left(B,\left\{u\right\}\right)$. 
Since $B$ is a filter, $e\left(\left\{u\right\}, U\right) = e\left(\left\{u\right\}, P\right)$, 
and since $u$ is minimal in $U$, there holds 
$e\left(B,\left\{u\right\}\right) = e\left(P,\left\{u\right\}\right)$. Consequently,
the change in the cut's size is
\[e\left(\left\{u\right\}, U\right)-e\left(B,\left\{u\right\}\right)= e\left(\left\{u\right\}, P\right)-e\left(P,\left\{u\right\}\right)\]
The assumed maximality of $(B,U)$'s, and 
proposition \ref{prop:continuity}, imply the first item.

The same argument, moving a maximal element of $B$ to $U$ completes the proof.
            \end{proof}
        \end{proposition}
\begin{rem}
This proposition yields the claim about the computational
complexity made in theorem \ref{thm:main}.
\end{rem}

\section{Proof of the main result}

The following clearly implies theorem \ref{thm:main}.

\begin{theorem}\label{thm:cheaper}
Let $(P,>)$ be a finite poset. Let $(B,U)$
be a partition of $P$, where
$B \supseteq \Delta_P$
contains all deficit elements of $P$, and 
$U\supseteq \Sigma_P$ contains all
its surplus elements. The balanced
elements are arbitrarily partitioned
between $B$ and $U$. Then
\[e\left(B, U\right) \ge \frac{e\left(P, P\right)}{2}.\]
\end{theorem}

\begin{proof} (Theorem \ref{thm:cheaper}) 
We prove the inequality for one such cut, without specifying
where the balanced elements are located.
By corollary \ref{cor:equiv} this implies the inequality for all such cuts.\\
The proof is by induction on $e(P,P)$. When $P$ is an antichain and
$e(P,P)=0$, the claim holds vacuously.

We start with several easy observations:

As usual, we say that $x<y$ is a {\em cover} relation in $(P,>)$ if
there is no $z\in P$ with $x<z<y$. This is denoted $x\prec y$.
We need another easy observation:

\begin{proposition}\label{prop:eliminate}
Let $x<y$ be two elements of a poset $(P,>)$. Consider the binary relation
that is obtained by eliminating the single relation $x<y$ from $(P,>)$.
The resulting binary relation is a partial 
order iff $x\prec y$.
\end{proposition}

Let $C=(u_1 > u_2 >\ldots > u_t)$ be a longest chain in $(P,>)$.
The following facts are easy to verify, using proposition \ref{prop:continuity}:
\begin{itemize}
\item 
For all $j<t$ there holds $u_{j+1}\prec u_j$.
\item 
$u_1\in \Sigma_P$ is a surplus element.
\item 
$u_t\in\Delta_P$ is a deficit element.
\item 
Either of the following two cases occurs:
\begin{enumerate}
\item
{\bf Case I:}
$u_{\beta}$ is balanced for some index $t>\beta>1$, while 
$u_{\alpha}$ is surplus, resp.\ deficit whenever $\alpha < \beta$ resp.\ $\alpha > \beta$.
\item
{\bf Case II:}
For some index $\beta>1$ it holds that
$u_{\alpha}$ is surplus, resp.\ deficit whenever $\alpha < \beta$ resp.\ $\alpha \geq \beta$.
\end{enumerate}
\end{itemize}

In both cases we create a poset $Q$ by removing one or two 
cover relations from $(P,>)$. The inductive hypothesis applies to $Q$,
since $e(Q,Q)<e(P,P)$. An application of the induction hypothesis yields
a large cut $(V,W)$ in $Q$, a slight modification of which yields the
cut $(B,U)$ alluded to in the theorem.

Notation for {\bf Case I}: $b=u_{\beta+1}\prec a=u_{\beta}\prec w=u_{\beta-1}$,
where $b, a$ and $w$ are deficit, balanced, and surplus in this order. The poset $Q$
is obtained by removing the cover relations $b\prec a$ and $a\prec w$ from $P$.
Notice that $a$ is balanced in $Q$ as well. For $b$ and $w$ the possibilities are as follows:
\begin{enumerate}
\item 
$b$ remains deficit and $w$ remains surplus.
\item 
$b$ remains deficit and $w$ becomes balanced.
\item 
$w$ remains surplus and $b$ becomes balanced.
\end{enumerate} 

By induction we find a cut $(V,W)$ in $Q$, as in theorem \ref{thm:main}. Namely,
\[ e_Q\left(V, W\right)\ge \frac{e_Q\left(Q, Q\right)}{2}, \]

where $V\supseteq \Delta_Q$ contains all deficit elements, and $W\supseteq\Sigma_Q$
contains all surplus elements.
\begin{itemize}
\item 
Case $I_1$: $\Delta_P=\Delta_Q$ and $\Sigma_P=\Sigma_Q$. Let $(B,U):=(V,W)$
be our cut in $P$.
\item 
Case $I_2$: $\Delta_P=\Delta_Q$, $\Sigma_P=\Sigma_Q\cup\{w\}$. Let 
$(B,U):=(V\setminus\{w\},W\cup\{w\})$ be our cut in $P$.
\item 
Case $I_3$: $\Delta_P=\Delta_Q\cup\{b\}$, $\Sigma_P=\Sigma_Q$. Let 
$(B,U):=(V\cup\{b\},W\setminus\{b\})$ be our cut in $P$.
\end{itemize}

In all cases the parts of the $P$-cut contain $\Delta_P$ resp.\ $\Sigma_P$,
as required. To verify inequality (\ref{eq:main}),
note that $e_P(P,P)=e_Q(Q,Q)+2$ due to the two removed cover relations.
Also, $e_P(B,U) = e_Q(V,W)+1$ since exactly one of the relations removed is 
from $B$ to $U$.

\vspace{0.2in}

And for {\bf Case II} we denote: $b=u_{\beta}\prec w=u_{\beta-1}$ for $b$, resp.\ $w$
a deficit, resp.\ surplus element.
The poset $Q$
is obtained by removing the cover relation $b\prec w$ from $P$.
For $b$ and $w$ the possibilities are as follows:
\begin{enumerate}
\item 
$b$ remains deficit and $w$ remains surplus.
\item 
$b$ remains deficit and $w$ becomes balanced.
\item 
$w$ remains surplus and $b$ becomes balanced.
\item 
$w$ and $b$ both become balanced.
\end{enumerate} 

Here $e_P(P,P)=e_Q(Q,Q)+1$ due to the removed cover relation, so we may
apply induction and find a cut $(V,W)$ in $Q$, as in theorem \ref{thm:main}. Namely,
\[ e_Q\left(V, W\right)\ge \frac{e_Q\left(Q, Q\right)}{2}, \]
where $V\supseteq \Delta_Q$ contains all deficit elements, and $W\supseteq\Sigma_Q$
contains all surplus elements.\\
We want to modify the $(V,W)$ cut in $Q$ to a cut $(B,U)$ in $P$, so that
$B\supseteq \Delta_P$, $U\supseteq\Sigma_P$, and $e_P(B,U) = e_Q(V,W)+1$.
This can be accomplished as follows:
\begin{itemize}
\item 
Case $II_1$: $B:= V$, $U:=W$.
\item 
Case $II_2$: $B:= V\setminus\{w\}$, $U:=W\cup\{w\}$.
\item 
Case $II_3$: $B:= V\cup\{b\}$, $U:=W\setminus\{b\}$.
\item 
Case $II_4$: $B:= (V\setminus\{w\})\cup\{b\}$, $U:=(W\setminus\{b\})\cup\{w\}$.
\end{itemize}
\end{proof}

\printbibliography
\end{document}